\documentclass[a4paper]{article}
\usepackage[a4paper,top=3cm,bottom=2cm,left=3cm,right=3cm,marginparwidth=1.75cm]{geometry}
\usepackage[english]{babel}
\usepackage{amsmath}
\usepackage{amssymb}
\usepackage{amsthm}
\usepackage{bbm} 
\usepackage{marvosym}
\usepackage{graphicx}
\usepackage{xcolor}
\usepackage{array}
\usepackage{subcaption}
\usepackage{authblk}

\newtheorem{theorem}{Theorem}[section]
\newtheorem{corollary}[theorem]{Corollary}
\newtheorem{lemma}[theorem]{Lemma}
\newtheorem{definition}[theorem]{Definition}
\newtheorem{proposition}[theorem]{Proposition}

\DeclareMathOperator{\cat}{cat}
\DeclareMathOperator{\dist}{dist}
\DeclareMathOperator{\median}{med}
\let\epsilon\varepsilon
\let\phi\varphi

\newcommand{\rp}[1]{\mathbb{RP}^#1}
\newcommand{\kras}{\mathsf{kr}}
\newcommand{\hs}{\mathsf{hs}}

\title{The homotopy significant spectrum compared to the Krasnoselskii spectrum}
\author{S.J. Fokma}
\author{J.W. Portegies}

\affil{Eindhoven University of Technology\\s.j.fokma@student.tue.nl, j.w.portegies@tue.nl}

\begin{document}

\maketitle

\begin{abstract}
How to generalize the concept of eigenvalues of quadratic forms to eigenvalues of arbitrary, even,
homogeneous continuous functionals, if stability of the set of eigenvalues under small perturbations is required? We compare two possible generalizations, Gromov's homotopy significant spectrum and the Krasnoselskii spectrum. We show that in the finite dimensional case, the Krasnoselskii spectrum is contained in the homotopy significant spectrum, but provide a counterexample to the opposite inclusion. Moreover, we propose a small modification of the definition of the homotopy significant spectrum for which we can prove stability. Finally, we show that the Cheeger constant of a closed Riemannian manifold corresponds to the second Krasnoselskii eigenvalue.
\end{abstract}

\section{Introduction}
Eigenfunctions and eigenvalues of operators are used by important tools to identify the most relevant components of a mechanical system, to remove noise from a picture, or to learn a manifold close to a dataset. 
For instance, to denoise an image, one can take the image as an initial condition and run the heat flow, i.e. the steepest descent of the Dirichlet energy $\int |\nabla u|^2$, for a short time. By doing so, eigenfunctions with large eigenvalues are damped. As another example, the manifold-learning algorithms Diffusion maps \cite{Coifman-Diffusion-2006} and Eigenmaps \cite{Belkin-Laplacian-2003} use eigenfunctions of a Laplace operator to map a high-dimensional dataset to a lower-dimensional space.

The Diffusion maps and Eigenmaps algorithms implicitly use that the eigenvalues and eigenfunctions of the Laplace operator defined on a manifold give away geometric information about the manifold itself. For instance, by the Weyl law, the asymptotic growth of the eigenvalues exposes the dimension and the volume of the manifold. 

The relevant operators in the above examples are symmetric and linear, but sometimes there are good reasons to look at nonlinear operators instead. For instance, (linear) heat flow on pictures blurs edges, although correct identification of the edges is important for medical applications \cite{Rudin-Nonlinear-1992, Zeune-Multiscale-2017, Duits-Total-2019}. It may then be better to follow the steepest descent of the total variation functional $\int |\nabla u|$ instead, which is non-linear. As another example, for spaces and datasets that locally do not look Euclidean, the generalization of the Laplace operator itself becomes nonlinear (see e.g. \cite{Shen-non-linear-1998} for the case of Finsler manifolds). Still, counterparts of eigenvalues and eigenfunctions may capture important information about the dataset.

The question is then, however, how to properly generalize eigenvalues and eigenfunctions to such a nonlinear context. Below we will motivate different generalizations. The objective of the article will be to compare two of them.

\subsection*{Eigenvalues as critical values}

If $E$ and $F$ are two quadratic forms on a finite-dimensional vector space $V$, and $F$ is positive-definite, then the spectral theorem gives that there exists a basis of eigenvectors $v_1, \dots, v_n$ in which both $E$ and $F$ are diagonal. The asymmetry between $E$ and $F$ can be removed if the dimension of $V$ is at least $3$: in that case $F$ no longer needs to be positive-definite, but $E$ and $F$ should be such that they only vanish simultaneously at the origin. This result is a bit more involved, but a short and elegant proof is due to Calabi \cite{Calabi-Linear-1964}.

If $F$ is positive-definite it induces an inner product on $V$ and the eigenvectors and eigenvalues of $E$ with respect to $F$ are exactly the critical points and values of the normalized energy
\[
\mathcal{E}(u) := \frac{E(u)}{F(u)}.
\]
Again, a more symmetric approach looks at the map $u \mapsto [E(u): F(u)]$ to projective space. As it is invariant under multiplication of the input by a scalar, it induces a map $f$ from projective to projective space. The eigenvectors are exactly the critical points of this map.

This suggests the following generalization if $E$ and $F$ are even, $\alpha$-homogeneous, smooth functions, only vanishing simultaneously at the origin. We can then define eigenvectors of the pair $(E,F)$ as critical points of the $0$-homogeneous normalized functional
\[
f([u]) = [E(u) : F(u)].
\]
Critical points of this functional obey the equation
\[
F(u) DE (u) = E(u) DF(u),
\]
where $D$ denotes the derivative. As a sidenote, the homogeneity of $E$ and $F$ makes their derivative in a point $u$ in the direction of $u$ easy to evaluate. Hence it follows that if a function $u$ satisfies 
\begin{equation}
\label{eq:ev2}
\nu DE(u) = \mu DF(u),
\end{equation}
then necessarily $[\mu, \nu] = [E(u), F(u)]$. In other words, if one would prefer $(\ref{eq:ev2})$ as a definition of eigenvectors and eigenvalues, rather than critical points and values of $\mathcal{E}$, then one would not obtain more eigenvectors.

In practical situations, the functionals $E$ and $F$ may not be smooth, and in that case one would need to find proper generalizations, for instance replacing derivatives by subderivatives. Another generalization uses the Morse-theory link between critical points of a Morse function and the topology of sublevel sets. This will be our perspective below.

\subsubsection*{Two examples}

Let us give two examples. The first is when $E$ is the total variation functional and $F$ is the $L^2$ norm. These eigenfunctions play a role in the denoising of images. In this particular case, there is even a spectral decomposition associated to the eigenfunctions. Moreover, eigenfunctions form characteristic building blocks in the $L^2$-steepest descent of the total variation \cite{Gilboa-Total-2014, Burger-Spectral-2016, Zeune-Multiscale-2017}.

The second example is the Cheeger constant, which has versions for various geometric objects. In this case, $E$ is again the total variation functional, and $F$ is the $L^1$ functional. In general, a small Cheeger constant expresses that a geometric object can be cut into two large pieces with a small cut. This is an important characteristic for graphs and networks. For manifolds, the Cheeger constant can be expressed as follows \cite[Remark 9.3]{Ambrosio-New-2017}, 
\[
h_1 := \inf_{u \in L^1 \setminus \{ 0 \}, u \text{ not constant} } \frac{TV(u)}{ \inf_{a \in \mathbb{R}} \int | u - a| } 
\]
where $TV(u)$ denotes the total variation of $u$ (see \ref{ap:cheeger} for a precise definition). The Cheeger constant can also be written as
\[
h_1 = \inf\{ TV(u) \ | \ u \in L^1(M), \|u\|_1 = 1, \median u = 0 \}
\]
In \ref{ap:cheeger} we show that the Cheeger constant of a closed Riemannian manifold can be seen as a nonlinear eigenvalue. We believe that this result is known, but decided to include our proof as it is rather short and is a nice illustration of the concepts introduced in the article. For Cheeger constants of open sets in $\mathbb{R}^n$, an analogous result follows by combining results by Parini \cite{Parini-Second-2010} and Littig and Schuricht \cite{Littig-Convergence-2014}. In the context of connected graphs, Chang gave a characterization of the Cheeger constant as a nonlinear eigenvalue \cite[Theorem 5.12 and Theorem 5.15]{Chang-Spectrum-2016}.

\subsubsection*{Stability}

Summarizing so far, we could generalize eigenvalues and eigenfunctions to critical values and critical points of a function $f$ defined on projective space. However, considering \emph{all} critical points has the disadvantage that small perturbations to $f$ can lead to large changes to the set of eigenvalues: in other words, the spectrum defined in this way is unstable under perturbations. 

For instance if we want to use spectra to analyze datasets, we would want to use methods that are stable in the sense that if the data changes slightly, they would give similar results. In a broader sense, we are often interested in the continuity of spectra under convergence of geometric objects in some topology. In \cite{Ambrosio-New-2017} Ambrosio and Honda showed for instance the continuity of the Cheeger constant in a class of geometric objects with generalized Ricci curvature lower bounds. In \cite{Ambrosio-Continuity-2018} the authors aimed to show continuity of eigenvalues for spaces that are not locally Euclidean. For such an application, a stable spectrum is crucial.

\subsection*{Homotopy significant spectrum}
Stability is one motivation\footnote{Although stability is a motivation, we currently do not see how to prove it for the original definition. We will come back to this point below.} for Gromov's definition of the {\bf homotopy significant spectrum} \cite{Gromov2015}, see also his Pauli Lectures in 2009 at ETH (which are available online \cite{Gromov-Pauli-2009}). This spectrum is defined for continuous, real-valued functions $E$ on a topological space $\Phi$. A value $a \in \mathbb{R}$ is in the homotopy significant spectrum of such an $E$ if the homotopy type of the sublevel set $\Phi_{\leq t} := \{ u \in \Phi \ | \ E(u) \leq t\}$ changes `significantly' if $t$ passes through $a$. Precisely, the value $a$ is in the homotopy significant spectrum if there does not exist a homotopy $H:\Phi_{\leq a}\times[0,1]\rightarrow \Phi$ such that $H(x,0)=\mathbbm{1}_{\Phi_{\leq a}}(x)$ for all $x\in\Phi_{\leq a}$, and $H(\Phi_{\leq a},1)\subset\Phi_{<a}$. The adjective \emph{significant} is reflected in that the homotopy $H$ is allowed to take values in \emph{all} of the ambient space $\Phi$, and not just in the sublevel set $\Phi_{\leq a}$.

\subsection*{Krasnoselskii spectrum}
The homotopy significant spectrum was introduced to lead to \emph{stable} eigenvalues, but it does not give a procedure to find the eigenvalues, or to index them in some way. That is, we do not know how to write down a bijective mapping from $\mathbb{N}$ or $\mathbb{R}$ to the set of homotopy significant eigenvalues, and in particular we do not know if the set of homotopy significant eigenvalues is countable. 

In contrast, in the linear case, the Courant-Fisher minmax principle \emph{does} provide a procedure to find eigenvalues: the $k$th eigenvalue is given by
\begin{equation}
\lambda_k = \inf_{\substack{V \subset H  \\ \dim V \geq k}} \sup_{u \in S(V)} E(u),
\end{equation} 
where the infimum is over all linear subspaces $V$ of the Hilbert space $H$, and $S(V)$ denotes the unit sphere in $V$.

Surely, in a nonlinear setting one can also use minmax techniques to find critical points, although the definition needs to be adapted slightly. The linearity condition on $V$ is too restrictive, but we can replace it by an optimization over closed symmetric subsets of the sphere, denoted by $\mathcal{V}(S(H))$. This results in
\begin{equation}
\label{eq:lambda-krasnoselskii}
\lambda_k = \inf_{\substack{A \in \mathcal{V}(S(H)) \\ \dim A \geq k}} \sup_{u \in A} E(u).
\end{equation} 
This definition only needs a good generalization of \emph{dimension}. This way of defining eigenvalues was also put forward by Gromov \cite{Gromov-Dimension} and was used as a definition for eigenvalues of the Laplace operators on Finsler manifolds by Zhongmin Shen \cite{Shen-non-linear-1998}.

Different concepts of dimension could a priori lead to different sets of eigenvalues. In this article, we will concentrate on one concept of dimension: the Krasnoselskii genus (which we will explain in the next section). We call $\lambda_k$ as defined by (\ref{eq:lambda-krasnoselskii}) the $k$th eigenvalue in the Krasnoselskii spectrum and denote it by $\kras_k$, if we give `$\dim$' the interpretation of the Krasnoselskii genus.

\medskip

This Krasnoselskii spectrum indeed allows for continuity proofs of spectra of operators: in \cite{Ambrosio-Continuity-2018} the authors show that the values in the Krasnoselskii spectrum are continuous with respect to measured Gromov-Hausdorff convergence of the underlying geometric objects. Similarly, the Cheeger constant is in fact the second Krasnoselskii eigenvalue, and the result by Littig and Schuricht \cite{Littig-Convergence-2014} can be seen as a stability result.

\subsection*{Comparison and stability}

In this article we investigate how the homotopy significant spectrum and the Krasnoselskii spectrum relate to each other. We show that the Krasnoselskii spectrum is always contained in the homotopy significant spectrum, but that in general they are not equal.

We conclude with a discussion on the stability of the homotopy significant spectrum. So far, we are not aware of a result on the stability of the spectrum for all continuous functions, and we are only able to prove its stability for Morse functions. Hence we propose an alternative definition of the homotopy significant spectrum, the \emph{weak homotopy significant spectrum}, of which its stability for continuous functions follows more directly.

\section{The Krasnoselskii genus and the Lusternik-Schnirelmann subspace category}
Before we can actually start comparing the two spectra, we will need to go over some properties of the Krasnoselskii genus and the closely related Lusternik-Schnirelmann subspace category. 

For a closed set $C$ in a Banach space $B$, let $\mathcal{V}(C)$ be defined as 
\begin{equation}
\mathcal{V}(C)= \{A\subset C | A \text{ closed and symmetric}\}
\end{equation} in which a symmetric subset is a subset $A\subset B$ which satisfies $A=-A$.
\begin{definition} \label{def:kgenus}
Let $B$ be a Banach space and $A\in\mathcal{V}(B)$. Then we define the Krasnoselskii genus $\gamma(A)$ of non-empty $A$ as the smallest integer $m \in \mathbb{N}$ such that there exists an odd, continuous function $h: A \to \mathbb{R}^m \backslash\{0\}$. If no such $m$ exists, we say $\gamma(A)= \infty$.
If $A=\emptyset$, then we define $\gamma(A)=0$.
\end{definition}
For later use, we will list a few general properties of the Krasnoselskii genus \cite[Prop. 5.4]{Struwe}.

\begin{proposition} \label{th:propgenus}
Let $B$ be a Banach space, and let $A,A_1,A_2\in\mathcal{V}(B)$. Then the following hold:
\begin{enumerate}
\item $\gamma(A)\geq0$, and $\gamma(A)=0$ if and only if $A=\emptyset$ (positive-definiteness);
\item $\gamma(A_1)\leq \gamma(A_2)$ if $A_1\subset A_2$  (monotonicity);
\item $\gamma(A_1\cup A_2)\leq \gamma(A_1)+\gamma(A_2)$ (subadditivity);
\item if $A$ is compact and $0\notin A$, then $\gamma(A)<\infty$ and there exists a neighbourhood $N$ of $A$ in $B$ such that $\overline{N}\in\mathcal{V}(B)$ and $\gamma(A)=\gamma(\overline{N})$.
\end{enumerate}
\end{proposition}

Since we will be mostly working with subsets of the unit sphere in a Banach space $B$, the following proposition concerning their genus will also be useful.

\begin{proposition} \label{th:genusS}
Let $B$ be an $n$-dimensional Banach space. Then the unit sphere $S$ has Krasnoselskii genus $n$, and the genus of all proper, closed and symmetric subsets of $S$ is strictly smaller than $n$.
\end{proposition}
\begin{proof}
The first part of the statement follows immediately from Proposition 5.2 in \cite{Struwe}.

For the second part assume there exists a proper, closed and symmetric subset $A$ of $S$ with $\gamma(A)=n$. Assume without loss of generality that $A$ does not contain the points $(0,\dotsc,0,\pm1)$. Define $p:A\rightarrow \mathbb{R}^{n-1}$ as $p(x_1,\dotsc,x_n)=(x_1,\dotsc,x_{n-1})$. Then $p$ is odd, continuous and $0\notin p(A)$. This constitutes a contradiction since $p$ maps to $\mathbb{R}^{n-1}$ and $n-1<\gamma(A)$.
\end{proof}

We say that a function $f$ is even if it satisfies $f(x)=f(-x)$ for all in $x$ in the domain of $f$. Using the Krasnoselskii genus we can pose a definition for eigenvalues for even and continuous functions. 

\begin{definition}
Let $B$ be a Banach space with unit sphere $S$ and let $f:S\rightarrow\mathbb{R}$ be an even and continuous function. Then we define for $k\in\{1,\dotsc,n\}$, if B is $n$-dimensional, or for $k \in \mathbb{N}$, if $B$ is infinite-dimensional, the eigenvalues\begin{equation} 
\kras_k = \inf_{\substack{A \in\mathcal{V}(S) \\ \gamma(A)\geq k}} \sup_{x \in A} f(x).
\end{equation}
We will refer to these eigenvalues as the Krasnoselskii spectrum of $f$.
\end{definition}
Note that it follows from this definition and Proposition \ref{th:genusS} that $\kras_1=\inf_{x\in S^n} f(x)$ and $\kras_{n+1}=\sup_{x\in S^n} f(x)$.

Because we also want to consider the functions on the real projective space induced by even functions on the sphere, we will need the Lusternik-Schnirelmann subspace category.

\begin{definition}
Let $\Phi$ be a topological space and let $A\subset\Phi$ be a closed subset. The Lusternik-Schnirelmann subspace category of $A$ in $\Phi$, denoted $\cat_\Phi A$, is the smallest integer $m$ such that $A$ is covered by closed sets $C_1,\dotsc, C_m$ which are contractible in $\Phi$. If no such covering exists, we write $\cat_\Phi A = \infty$. If $A=\emptyset$, then we define $\cat_\Phi A=0$.
\end{definition}

We would like to emphasize that in this definition the category makes use of closed instead of open sets, even though the latter is often encountered in literature nowadays. Here, we need to make use of the closed subsets to be able to relate the Krasnoselskii genus and category as in Lemma \ref{th:geniscat} below. Before we prove the lemma, we note however that since they are both examples of an index \cite[p. 99-100]{Struwe}, the category has similar properties as the Krasnoselskii genus \cite[Prop. 5.13]{Struwe}.

\begin{proposition} \label{th:proplscat}
Let $\Phi$ be a topological space and let $A,A_1,A_2$ be closed subsets of $\Phi$. Then the following hold:
\begin{enumerate}
\item $\cat_\Phi (A)\geq0$, and $\cat_\Phi (A)=0$ if and only if $A=\emptyset$ (positive-definiteness);
\item $\cat_\Phi (A_1)\leq \cat_\Phi (A_2)$ if $A_1\subset A_2$ (monotonicity);
\item $\cat_\Phi (A_1\cup A_2)\leq \cat_\Phi(A_1)+\cat_\Phi(A_2)$ (subadditivity);
\end{enumerate}
\end{proposition}

Let $B$ be a Banach space and $A\subset B \setminus\{0\}$ a symmetric subset. Denote with $A/\mathbb{Z}_2$ the image of $A$ under the quotient map which identifies $\{u,-u\}$ for all $u\in A$. Under this identification, Lemma \ref{th:geniscat} shows that the Krasnoselskii genus and the category are the same for elements of $\mathcal{V}(S)$.

\begin{lemma} \label{th:geniscat}
Let $B$ be an $n$-dimensional Banach space with unit sphere $S$, and let $A\in\mathcal{V}(S)$. Then $\gamma(A)=\cat_{S/\mathbb{Z}_2}(A/\mathbb{Z}_2)$.
\end{lemma}
\begin{proof}
From \cite[Th. 3.7]{Rabinowitz1973} we know that $\gamma(A)=\cat_{(B\setminus\{0\})/\mathbb{Z}_2} (A/\mathbb{Z}_2)$ since $A$, as a closed subset of the unit sphere in the finite-dimensional Banach space $B$, is compact. What remains is to prove that for $A\in\mathcal{V}(S)$ the subspace category in $(B\setminus \{0\})/\mathbb{Z}_2$ equals the subspace category in $S/\mathbb{Z}_2$.

The inequality $\cat_{(B\setminus \{0\})/\mathbb{Z}_2} (A/\mathbb{Z}_2) \leq \cat_{S/\mathbb{Z}_2} (A/\mathbb{Z}_2)$ follows immediately from the fact that a closed, contractible covering of $A/\mathbb{Z}_2$ in $S/\mathbb{Z}_2$ remains as such when considered in $(B\setminus \{0\})/\mathbb{Z}_2$.

Now suppose $\cat_{(B\setminus \{0\})/\mathbb{Z}_2} (A/\mathbb{Z}_2)=m$. Let $\{C_1,\dotsc, C_{m}\}$ be a covering of $A/\mathbb{Z}_2$ consisting of closed and contractible sets in $(B\setminus \{0\})/\mathbb{Z}_2$. Define $V_i=C_i\bigcap S/\mathbb{Z}_2$ for each $i\in\{1,\dotsc,m\}$. All $V_i$ are closed and together they form a cover of $A/\mathbb{Z}_2$. These $V_i$ are also contractible, because if $H:C_i\times[0,1]\rightarrow(B\setminus\{0\})/\mathbb{Z}_2$ contracts $C_i$ to a point $x_0\in(B\setminus\{0\})/\mathbb{Z}_2$, then $H/\|H\|\rvert_{V_i}:V_i\times[0,1]\rightarrow S/\mathbb{Z}_2$ contracts $V_i$ to $x_0/\|x_0\|\in S/\mathbb{Z}_2$. Hence $\cat_{(B\setminus \{0\})/\mathbb{Z}_2} (A/\mathbb{Z}_2) \geq \cat_{S/\mathbb{Z}_2} (A/\mathbb{Z}_2)$.
\end{proof}

We now introduce eigenvalues in which the category plays the role of the essential dimension of subsets.

\begin{definition}
\label{def:hsk}
Let $B$ be a Banach space, let $\Phi$ be the corresponding projective space, and let $f:\Phi \rightarrow\mathbb{R}$ be a continuous function.
Then we define for $k\in\{1,\dotsc,n\}$, if B is $n$-dimensional, or for $k \in \mathbb{N}$, if $B$ is infinite-dimensional, the eigenvalues\begin{equation} 
\hs_k = \inf_{\substack{A \subset \Phi \ \mathrm{ closed} \\ \cat_\Phi(A)\geq k}} \sup_{x \in A} f(x).
\end{equation}
\end{definition}

The next proposition shows that the eigenvalues $\kras_k$ and $\hs_k$ are exactly the values at which the Krasnoselskii genus respectively the category of the sublevel sets change.

\begin{proposition} \label{pro:altcharkrhs}
Let $B$ be a Banach space, let $S$ be the corresponding unit sphere and $\Phi$ the corresponding projective space.
Let $f:S\rightarrow\mathbb{R}$ be an even and continuous function.
We have the following characterizations
\[
\kras_k = \inf\{t\in \mathbb{R}|\gamma(S_{\leq t})\geq k\}
\]
and
\[
\hs_k=\inf\{t\in \mathbb{R}|\cat_{\Phi}(\Phi_{\leq t})\geq k\}
\]
\end{proposition}
\begin{proof}
We only show the statement for the Krasnoselskii spectrum, as the proof of the other statement is completely analogous. 
Define 
\[
m_k := \inf\{t\in \mathbb{R}|\gamma(S_{\leq t})\geq k\}
\]

Fix $k\in\{1,\dotsc,n\}$ if $B$ is $n$-dimensional or $k \in \mathbb{N}$ otherwise. We start by proving $\kras_k\leq m_k$. Let $m > m_k$. Note that the sublevel $S_{\leq m}$ is closed, symmetric and that by the monotonicity of the genus in Proposition \ref{th:propgenus} it holds that $\gamma(S_{\leq m})\geq k$. Hence \begin{equation}
\kras_k=\inf_{\substack{A \in\mathcal{V}(S) \\ \gamma(A)\geq k}} \sup_{x \in A} f(x)
\leq \sup_{x\in S_{\leq m}}f(x) \leq m
\end{equation}
Since $m > m_k$ was arbitrary, it follows that $\kras_k \leq m_k$.

Next, we show that $m_k \leq \kras_k$. Denote $\sup_{x\in A} f(x)$ by $s_A$ for all $A\in\mathcal{V}(S)$. If we take $A\in\mathcal{V}(S)$ such that $\gamma(A)\geq k$, then the monotonicity of the Krasnoselskii genus tells us that $\gamma(S_{\leq s_A})\geq \gamma(A)=k$. By definition of $m_k$, it holds that $\sup_{x\in A} f(x)=s_A\geq m_k$. By taking the infimum over all such $A$, we find that $m_k \leq \kras_k$.
\end{proof}

As a consequence of Lemma \ref{th:geniscat}, we know that for functions on finite-dimensional projective space, the Krasnoselskii eigenvalues $\kras_k$ correspond to the eigenvalues $\hs_k$.

\begin{theorem}\label{th:ksinhss}
Let $f:\rp{n}\rightarrow\mathbb{R}$ be a continuous function, then $\kras_k=\hs_k$ for all $k\in\{1,\dotsc,n+1\}$.
\end{theorem}
\begin{proof}
This follows immediately from Lemma \ref{th:geniscat} and Proposition \ref{pro:altcharkrhs}.
\end{proof}

For functions defined on infinite-dimensional projective spaces, we are not sure if the values $\kras_k$ and $\hs_k$ always agree. However, the \emph{second} eigenvalue always agrees and has the following nice characterization in terms of a min-max formula over non-contractible loops.

\begin{lemma}
\label{le:characterization-second-eigenvalue}
Let $B$ be a Banach space, let $S$ be the corresponding unit sphere and let $\Phi$ be the corresponding projective space. Let $f: \Phi \to \mathbb{R}$ be continuous.

Then
\[
\kras_2 = \hs_2 =  \inf\left\{ \sup_{s \in [0,1]} f(\ell(s))  \ | \ \ell : [0, 1] \to \Phi \text{ non-contractible loop} \right\}
\]
\end{lemma}

\begin{proof}
By the proof of Theorem 3.7 in \cite{Rabinowitz1973}, it holds that if the Krasnoselskii genus of a closed, symmetric subset of the sphere equals $1$, the category of the corresponding set in projective space equals $1$ as well. Therefore, the inequality $\kras_2 \leq \hs_2$ always holds.

We will now show that $\hs_2$ is smaller than the infimum. Let $\ell: [0,1] \to \Phi$ be a non-contractible loop.
Then the category of $\ell([0,1])$ is larger than or equal to two. Using Definition \ref{def:hsk} we find that $\hs_2$ is less than or equal to 
\[
\inf\left\{ \sup_{s \in [0,1]} f(\ell(s))  \ | \ \ell : [0, 1] \to \Phi \text{ non-contractible loop} \right\}
\]
In particular, $\kras_2$ and $\hs_2$ are finite. 

We will now show that
\begin{equation}
\label{eq:kr2-inf-ineq}
\kras_2 \geq \inf\left\{ \sup_{s \in [0,1]} f(\ell(s))  \ | \ \ell : [0, 1] \to \Phi \text{ non-contractible loop} \right\}
\end{equation}
Take $\epsilon > 0$, then $\gamma (S_{\leq \kras_2+\epsilon/2})\geq2$. Let $(A_i)_{i\in I}\subset S$ for some index set $I$ be the connected components of $S_{< \kras_2+\epsilon}$ such that $S_{< \kras_2+\epsilon}=\bigsqcup_{i\in I}A_i$. Here we denote by $\bigsqcup$ the disjoint union.

We claim that there exists an $i\in I$ such that $A_i=-A_i$. To prove this, assume instead that $A_i\neq-A_i$ for all $i\in I$. We will derive a contradiction. Since $S_{< \kras_2+\epsilon}$ is symmetric, there exist $J_1,J_2\subset I$ such that $I=J_1 \sqcup J_2$ and such that for all $i\in J_1$ there exists a $j\in J_2$ such that $A_i=-A_j$. Define the map $h:S_{< \kras_2+\epsilon}\rightarrow \mathbb{R}\setminus \{0\}$ as $h(x)=1$ for all $x \in A_i$ such that $i\in J_1$ and $h(x)=-1$ for all $x\in A_i$ such that $i \in J_2$. Then $h$ is continuous, odd and maps to $\mathbb{R}\setminus\{0\}$. Hence $\gamma(S_{\leq \kras_2+\epsilon/2}) = 1$, which is in contradiction with $\gamma(S_{\leq \kras_2+\epsilon/2})\geq 2$. We conclude that there is an index $i\in I$ such that $A_i=-A_i$. We denote $A := A_i$. 

As $A$ is open and connected, it is also path-connected. 
Hence $A$ must contain for some $x_0\in A$ a path $P:[0,1]\rightarrow A$ with $P(0)=x_0$ and $P(1)=-x_0$. This path induces a loop $\ell:[0,1]\rightarrow\Phi_{\leq\kras_2+\epsilon}$ which is not null-homotopic in $\Phi$. Therefore
\[
\kras_2 + \epsilon \geq \inf\left\{ \sup_{s \in [0,1]} f(\ell(s))  \ | \ \ell : [0, 1] \to \Phi \text{ non-contractible loop} \right\}\]
and inequality (\ref{eq:kr2-inf-ineq}) follows as $\epsilon$ was arbitrary.
\end{proof}

\section{Inclusion of the Krasnoselskii spectrum}
To relate the Krasnoselskii spectrum to the homotopy significant spectrum we first relate the existence of a homotopy between sublevels to their category.
 
\begin{lemma}\label{th:homocategory}
Let $\Phi$ be a topological space and $X,Y\subset \Phi$. If there exists a homotopy $H:X\times[0,1]\rightarrow\Phi$ with $H(x,0)=\mathbbm{1}_X(x)$ for all $x\in X$ and $H(X,1)\subset Y$, then $\cat_\Phi (X) \leq \cat_\Phi (Y)$.
\end{lemma}
\begin{proof}
Assume that $\cat_\Phi (Y)=m$ and let $\{C_1,\dotsc,C_{m}\}$ be a collection of closed and contractible sets in $\Phi$ which cover $Y$. Define $V_i=H^{-1}(C_i,1)$ for each $i\in\{1,\dotsc,m\}$. All $V_i$ are then closed and together they form a cover of $X$. Furthermore, the $V_i$ are contractible because we can first apply the homotopy $H$ to $V_i$ and then use the contractibility of $C_i$. Hence $\cat_\Phi (X) \leq m$.
\end{proof}

\begin{corollary}\label{th:homocategoryeq}
Let $\Phi$ be a topological space and $Y\subset X\subset \Phi$. If there exists a homotopy $H:X\times[0,1]\rightarrow\Phi$ with $H(x,0)=\mathbbm{1}_X(x)$ for all $x\in X$ and $H(X,1)\subset Y$, then $\cat_\Phi (X) = \cat_\Phi (Y)$.
\end{corollary}

This corollary tells us in the context of two sublevels of a continuous function, that if we can bring the larger sublevel to the smaller sublevel via a homotopy, that these sublevels must be of the same category. More important for us is however the negative statement. This tells us that whenever the category of the sublevels increases we have encountered an eigenvalue in the homotopy significant spectrum. Hence the  eigenvalues $\hs_k$ are homotopy significant. 

\begin{corollary} \label{pr:hscat}
Let $B$ be a Banach space, let $\Phi$ the corresponding projective space and let $f : \Phi \to \mathbb{R}$ be a continuous function. Then the eigenvalues $\hs_k\in\mathbb{R}$ are homotopy significant.
\end{corollary}

In the case of a finite-dimensional Banach space, we can now conclude that the Krasnoselskii spectrum is contained in the homotopy significant spectrum.

\begin{corollary}\label{co:ksinhss}
Let $B$ be an $n$-dimensional Banach space, let $S$ be the corresponding unit sphere and let $f:S\rightarrow\mathbb{R}$ be a continuous and even function. Then the Krasnoselskii eigenvalues $\kras_k\in\mathbb{R}$ are homotopy significant.
\end{corollary}
\begin{proof}
This follows immediately from Theorem \ref{th:ksinhss} and Corollary \ref{pr:hscat}.
\end{proof}

\section{Equality of the spectra}
So far we have shown that the Krasnoselskii spectrum is contained in the homotopy significant spectrum for all functions defined on $\rp{n}$. The follow-up question is then if there are values of $n$ and functions $f$ for which the Krasnoselskii spectrum is not only contained in, but equals the homotopy significant spectrum. Note that homotopy significant eigenvalues which are not in the Krasnoselskii spectrum can only occur in the interval $(\hs_1,\hs_{n+1})$, as $\hs_1=\inf_{x\in\rp{n}}f(x)$ and $\hs_{n+1}=\sup_{x\in\rp{n}}f(x)$. Using the contractibility of sublevels of category 1, we prove that $\hs_1$ is the only homotopy significant eigenvalue below $\hs_2$ for all $n$. This immediately shows the equality of the spectra for $n=1$.

\begin{theorem}\label{th:belowm2}
Let $B$ be a Banach space, let $\Phi$ be the corresponding projective space, and let $f:\Phi \rightarrow\mathbb{R}$ be a continuous function.
There is nothing in the homotopy significant spectrum of $f$ besides $\hs_1$ below $\hs_2$.
\end{theorem}
\begin{proof}
Assume without loss of generality that $\hs_1\neq\hs_2$, otherwise the claim follows immediately. Take $t\in(\hs_1,\hs_2)$. By definition of both $\hs_1$ and $\hs_2$ we have that $\cat_{\Phi} (\Phi_{\leq t})=1$. Hence a set $A\subset\Phi$ exists such that $A$ is closed, contractible and $\Phi_{\leq t}\subset A$. In particular a homotopy $H:\Phi_{\leq t}\times[0,1] \rightarrow \Phi$ exists such that $H(x,0)=\mathbbm{1}_A(x)$ and $H(x,1)=x_0\in\Phi$. Since $\Phi$ is path connected we can assume that $x_0\in\Phi_{<t}$. Hence $t$ cannot be homotopy significant, so $\hs_1$ is the only eigenvalue in the homotopy significant spectrum below $\hs_2$. 
\end{proof}

\begin{corollary}\label{th:allgromov1}
Let $f:\rp1\rightarrow\mathbb{R}$ be a continuous function. The eigenvalues $\hs_1$ and $\hs_2$ form the homotopy significant spectrum of $f$.
\end{corollary}

For functions defined on $\rp2$ proving a similar statement requires more work, but also in this case the Krasnoselskii spectrum equals the homotopy significant spectrum. 

\begin{theorem}\label{th:allgromov2}
Let $f:\rp2\rightarrow\mathbb{R}$ a continuous function. The eigenvalues $\hs_1$, $\hs_2$ and $\hs_3$ form the homotopy significant spectrum of $f$.
\end{theorem}
\begin{proof}
Assume without loss of generality that $\hs_1\neq\hs_2\neq\hs_3$. It follows from Lemma \ref{le:characterization-second-eigenvalue} that every sublevel $\rp2_{\leq\hs_2+\epsilon}$ for $0<\epsilon<\hs_3-\hs_2$ contains a loop $\ell:[0,1] \to \rp2_{\leq\hs_2+\epsilon}$ which is not null-homotopic in $\rp2$.

If $t\in \mathbb{R}\setminus\{\hs_1,\hs_2,\hs_3\}$ is a homotopy significant eigenvalue, it must satisfy $\hs_2<t<\hs_3$ by Theorem \ref{th:belowm2}. Therefore fix $t\in (\hs_2,\hs_3)$ and subsequently $\epsilon\in\mathbb{R}$ such that $0<\epsilon<(t-\hs_2)/2$. We will show that we can bring the sublevel $\rp2_{\leq t}$ to the path $\ell$ in $\rp2_{\leq\hs_2+\epsilon}$, and therefore it cannot be homotopy significant.

Since $S^2_{\leq t}$ does not cover all of $S^2$ by Proposition \ref{th:genusS}, it can be squashed with a odd homotopy to an equator. This homotopy induces a homotopy on $\rp2$. Moreover, as the fundamental group of $\rp2$ is $\mathbb{Z}_2$, the image of the equator is homotopic to $\ell$. Composing the two homotopies, we get a homotopy that brings $\rp2_{\leq t}$ inside $\rp2_{\leq \hs_2 + \epsilon}$.
\end{proof}

\section{Inequality of the spectra}
\label{se:inequality-spectra}
We have now proven that for functions defined on the real projective line or plane, the Krasnoselskii spectrum equals the homotopy significant spectrum. It turns out that this result cannot be extended to the real projective space $\rp3$. To show this we will construct a function and some value $t\in\mathbb{R}$ for which both the $\leq t$-sublevel and the $<t$-sublevel have category 3, but the $\leq t$-sublevel cannot be brought to the $<t$-sublevel.

Specifically, we will construct a function in such a way that its $\leq t$-sublevel contains $\rp2$ and its $<t$-sublevel deformation retracts onto the connected sum of $\rp2$ and the torus $T$, denoted as $\rp{2}\# T$ and sometimes called Dyck's surface, see also Figure \ref{fig:D3}. The idea behind this choice is that $\rp2$ can only be included in a tubular neighbourhood of $\rp{2}\# T$ if this neighbourhood is large enough.

\begin{figure}[h]
\centering
\includegraphics[width=0.3\textwidth]{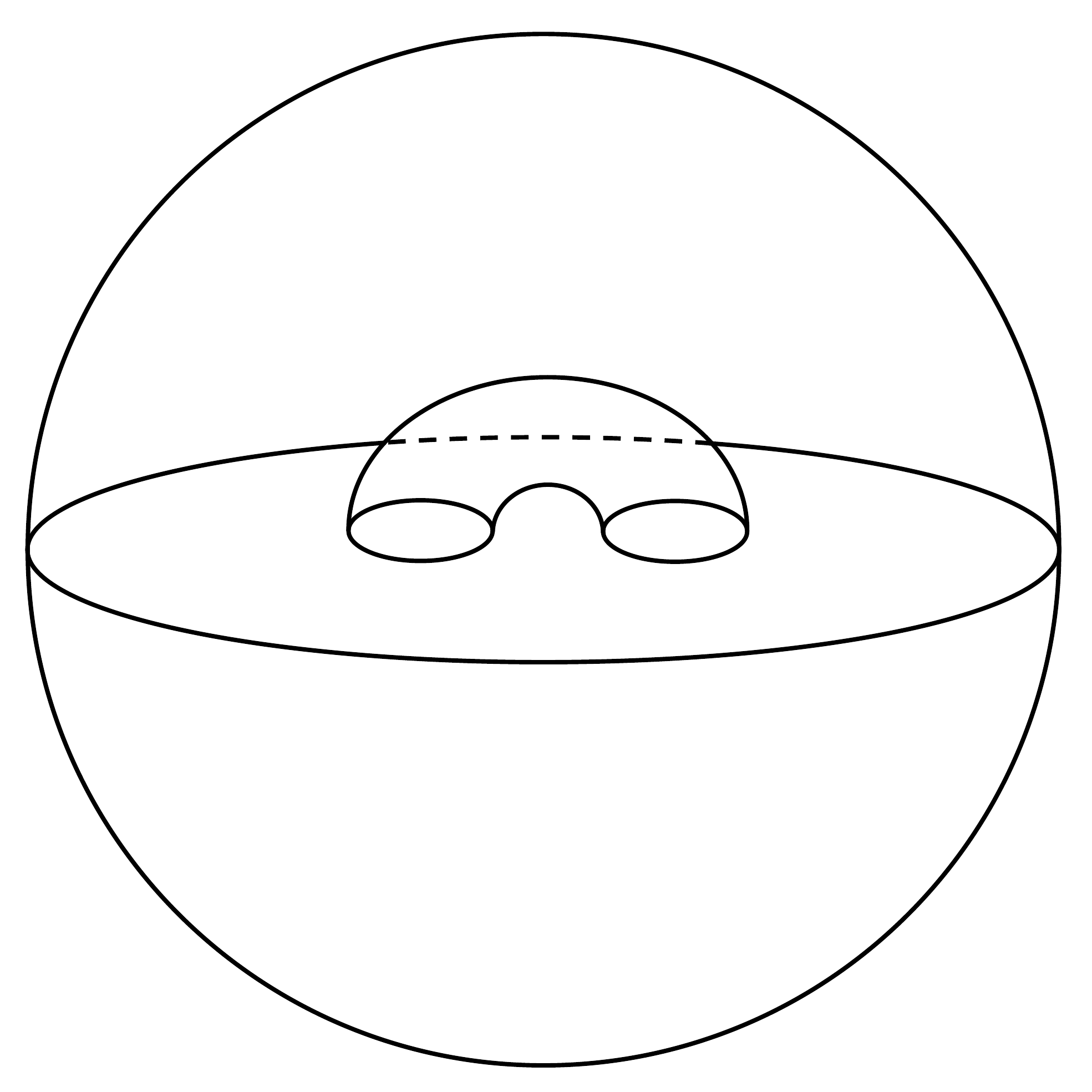}
\caption{$D^3/\sim$ with in the center $\rp{2}\#T$} \label{fig:D3}
\end{figure}

To show that we can indeed use this idea, we will first prove that the category of $\rp{2}\# T$ is 3. Following, we will prove that $\rp2$ cannot be brought to $\rp{2}\# T$. Throughout the remainder we will identify $\rp3$ with $D^3/\sim$ where $D^3=\{(x,y,z)\in\mathbb{R}^3|x^2+y^2+z^2\leq 1\}$ and the equivalence relation $\sim$ identifies antipodes on the boundary of $D^3$.

\begin{lemma}$\cat_{\rp{3}}(\rp{2}\#T)=3$. \label{th:catrpt3}
\end{lemma}
\begin{proof}
We first explain why it is enough to show that the complement of a small neighborhood of $\rp{2}\# T$ is contractible. 
Define for $\epsilon>0$ the tubular neighbourhood of $\rp{2}\# T$ in $\rp3$ as $T_\epsilon(\rp{2}\# T)=\{x\in\rp3|\dist(x,\rp{2}\# T)<\epsilon\}$. By Corollary \ref{th:homocategoryeq}, there exists an $\epsilon>0$ such that $\cat_{\rp{3}}\left(\overline{T_\epsilon(\rp{2}\#T)}\right)=\cat_{\rp{3}}(\rp{2}\#T)$. Then by the subadditivity of the category (see Proposition \ref{th:proplscat}), it holds that
\begin{equation}
\cat_{\rp{3}}(\rp{3})\leq \cat_{\rp{3}} \left(\overline{T_\epsilon(\rp{2}\#T)}\right) +\cat_{\rp{3}}(\rp{3}\setminus T_\epsilon(\rp{2}\#T)).
\end{equation}
By Proposition \ref{th:genusS} and Lemma \ref{th:geniscat} it holds that $\cat_{\rp{3}}(\rp{3})=4$. It then follows by again Proposition \ref{th:genusS} that $\cat_{\rp{3}} \left(\overline{T_\epsilon(\rp{2}\#T)}\right)\leq3$ as $\rp{2}\#T$ is a proper subset of $\rp{3}$. Hence, if we want to prove that $\cat_{\rp{3}} \left(\rp{2}\#T)\right) = \cat_{\rp{3}} \left(\overline{T_\epsilon(\rp{2}\#T)}\right) =  3$, it suffices to show that $\cat_{\rp{3}}(\rp{3}\setminus T_\epsilon(\rp{2}\#T))=1$. Therefore we show that $\rp{3}\setminus T_\epsilon(\rp{2}\#T)$ is contractible.

If we choose $\epsilon$ small enough, the subset of $D^3$ induced by $\rp3\setminus T_\epsilon(\rp{2}\# T)$ consists of two connected components: one containing the point $(0,0,1)$ and the other containing $(0,0,-1)$. 
We can make a strong deformation retract of the first connected component to the upper half sphere, and of the second component to the lower half sphere. We compose this homotopy by a symmetric homotopy that brings the upper half sphere (minus a neighborhood of the equator) to the point $(0,0,1)$ and the lower half sphere to the point $(0,0,-1)$. This second homotopy is continuous away from the equator. Combining the two homotopies and passing to projective space, we find that $\rp{3}\setminus T_\epsilon(\rp{2}\# T)$ is contractible.
\end{proof}

This implies that a sublevel which deformation retracts onto $\rp{2}\# T$, for instance a small enough tubular neighbourhood, also has category 3. We are now able to prove that $\rp2$ cannot be brought to $\rp{2}\# T$.

\begin{lemma}\label{th:rp2torp2t}
There does not exist a homotopy $H:\rp{2}\times[0,1]\rightarrow\rp{3}$ such that $H(x,0)=\mathbbm{1}_{\rp{2}}(x)$ for all $x\in\rp{2}$ and $H(\rp{2},1)\subset\rp{2}\#T$.
\end{lemma}
\begin{proof}
Let $i:\rp{2}\hookrightarrow\rp{3}$ and $j:\rp{2}\#T\hookrightarrow\rp{3}$ denote the inclusions of $\rp{2}$ and $\rp{2}\#T$ into $\rp{3}$, respectively. Assume a homotopy $G:\rp{2}\times[0,1]\rightarrow\rp{3}$ exists such that $G(x,0)=\mathbbm{1}_{\rp{2}}(x)$ for all $x\in\rp{2}$ and $G(\rp{2},1)\subset \rp{2}\#T$, then we will derive a contradiction. We denote the homotopy here with $G$ to avoid confusion with the homology groups.

Define $G_0:\rp{2}\rightarrow\rp{2}$ as $G_0(x)=G(x,0)$ for all $x\in\rp{2}$ and $G_1:\rp{2}\rightarrow\rp{2}\#T$ as $G_1(x)=G(x,1)$ for all $x\in\rp{2}$. Then since $i\circ G_0$ and $j\circ G_1$ are homotopic, it holds for the induced functions on the homology classes that $i_*=i_*G_{0*}=j_* G_{1*}$. By cellular homology, the function $i_*:H_2(\rp{2};\mathbb{Z}_2) \rightarrow H_2(\rp{3};\mathbb{Z}_2)$ is an isomorphism. Hence the map $G_1$ must be of $\mathbb{Z}_2$-degree $1$. However, $\rp2\#T$ has a (non-orientable rather than Krasnoselskii) genus strictly larger than $\rp2$ and so $G_{1}$ cannot exist (see for instance Lemma 8 in \cite{Burton-Finding-2017}). Therefore, the homotopy $G$ does not exist. 
\end{proof}

Our goal was to show that there exist functions on $\rp{3}$ of which the homotopy significant spectrum is larger than their Krasnoselskii spectrum. At this point, we have two spaces $\rp{2}$ and $\rp{2}\# T$ of category 3 and we know that we cannot bring the first into the second using a homotopy. Hence it remains to construct a function which has the two subspaces as sublevels. We will start by constructing a continuous function for which this holds.

\begin{theorem}\label{th:hsslargerks}
There exists a continuous function $f:\rp{3}\rightarrow\mathbb{R}$ such that its homotopy significant spectrum is strictly larger than its Krasnoselskii spectrum.
\end{theorem}
\begin{proof}
\begin{figure}[h]
\centering
\includegraphics[width=0.5\textwidth]{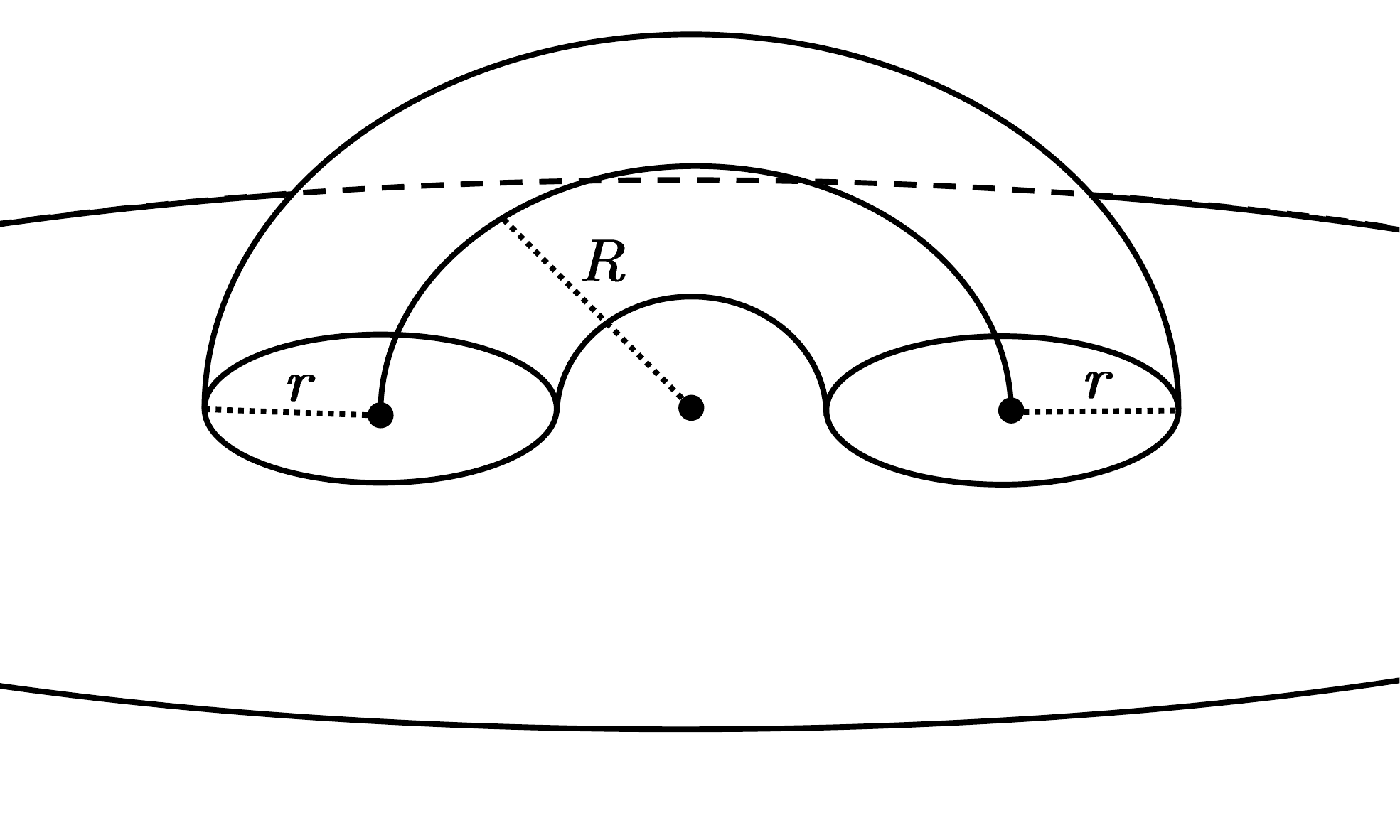}
\caption{Zoom-in view of where the torus $T$ is attached to $\rp2$ in $\rp{2}\# T$, with the radii 
$r$ and $R$ of the torus indicated} \label{fig:D3r}
\end{figure}
Consider the space $\rp{2}\# T$. Let $r>0$ denote the radius of the tube in the torus $T$ and let $R$ denote the radius from the center of the torus $T$ to the center of the torus tube as in Figure \ref{fig:D3r}. Take $R>3r$ and define the function $f:\rp{3}\rightarrow\mathbb{R}$ as $f(x)=\dist(x,\rp{2}\#T)$ for all $x\in\rp{3}$. Then we claim that the homotopy significant spectrum of $f$ contains an eigenvalue that is not contained in the Krasnoselskii spectrum of $f$. 

By the choice of $R$, the sublevel $\rp3_{\leq t}$ deformation retracts to $\rp{2}\#T=\rp3_{\leq 0}$ for all $t\in[0,r)$. However, since $\rp2\subset\rp3_{\leq r}$, Theorem \ref{th:rp2torp2t} tells us that a homotopy which brings $\rp3_{\leq r}$ to $\rp3_{<r}$ does not exist. Hence $r$ lies in the homotopy significant spectrum. However, since $\hs_3=0<r<\sup_{x\in\rp3}f(x)=\hs_4$, the value $r$ is not an element of the Krasnoselskii spectrum. 
\end{proof}

In this proof we had to make sure we could point out the exact value where the sublevel changed from something which deformation retracted onto $\rp{2}\# T$ to something which contained $\rp2$. One could think of cases where such a value cannot exactly be identified, but only two sublevels can be found where the larger one cannot be brought into the smaller one. The following lemma shows that in the case of Morse functions this is enough to detect the passing of an eigenvalue.

\begin{lemma} \label{th:hssintervalm}
Let $M$ be a smooth manifold and $F:M\rightarrow\mathbb{R}$ a Morse function. If for some $s,t\in\mathbb{R}$, where $s<t$, the sublevel $M_{\leq t}$ cannot be brought to $M_{\leq s}$ via a homotopy, the interval $(s,t]$ contains a homotopy significant eigenvalue. 
\end{lemma}
\begin{proof}
Let $c_1,\dotsc,c_m$ for some $m\in\mathbb{N}$ be all critical values of $F$ satisfying $s< c_1<c_2<\dotsc<c_m\leq t$. Since $F$ is Morse we can bring $M_{< c_1}$ to $M_{\leq s}$, and $M_{<c_{i+1}}$ to $M_{\leq c_i}$ for all $i\in\{1,\dotsc,m-1\}$. If $t\neq c_m$, we can bring $M_{\leq t}$ to $M_{\leq c_m}$ \cite[p. 14-20]{MilnorMorse}. It follows that one of the $c_i$ is homotopy significant.
\end{proof}
Note that variations of Lemma \ref{th:hssintervalm} can be proven similarly. In particular, as we will need this later, we would like to remark that if $M_{\leq t}$ cannot be brought to $M_{<s}$, the interval $[s,t]$ contains a homotopy significant eigenvalue.

\section{Stability}
We are not aware of a result on the stability of the homotopy significant spectrum for continuous functions. However, if we limit ourselves to Morse functions such a result follows directly.

\begin{lemma} \label{th:hssstable}
Let $M$ be a smooth manifold and $f:M\rightarrow\mathbb{R}$ a continuous real-valued function with a homotopy significant eigenvalue $e\in\mathbb{R}$. If $F:M\rightarrow\mathbb{R}$ is a Morse function and there exists a $\delta>0$ such that $\sup_{x\in M}|f(x)-F(x)|<\delta$, then $F$ has a homotopy significant eigenvalue $\widetilde{e}$ with $|e-\widetilde{e}|\leq\delta$.
\end{lemma}
\begin{proof}
If $e\in\mathbb{R}$ is an eigenvalue of $f$, this means that we cannot bring $f^{-1}(-\infty,e]$ to $f^{-1}(-\infty,e)$. Hence $F^{-1}(-\infty,e+\delta]$ cannot be brought to $F^{-1}(-\infty,e-\delta)$. The remark following Lemma \ref{th:hssintervalm} then tells us that the interval $[e+\delta,e-\delta]$ contains a homotopy significant eigenvalue.
\end{proof}

To be able to prove the stability of the homotopy significant spectrum for continuous functions, we propose to slightly change the definition of the homotopy significant spectrum to introduce the weak homotopy significant spectrum. It follows directly from the definition below that every homotopy significant eigenvalue is also weakly homotopy significant. However, for this weak homotopy significant spectrum its stability follows again almost directly.

\begin{definition}
Let $\Phi$ be a topological space and $f:\Phi\rightarrow\mathbb{R}$ a continuous real-valued function. The value $e\in\mathbb{R}$ is an eigenvalue in the weak homotopy significant spectrum of $f$ if for all  $t_1 > e$ and $t_2<e$ there does not exist a homotopy which brings $\Phi_{\leq t_1}$ to $\Phi_{\leq t_2}$.
\end{definition}

The stability of the weak homotopy significant spectrum can now be shown as in the proof of Lemma \ref{th:hssstable}. We need to combine it however with an adaptation of Lemma \ref{th:hssintervalm} to the weak homotopy significant spectrum. Hence we first prove this adaptation, and then the stability.  

\begin{lemma}\label{th:hssaltempty}
Let $\Phi$ be a topological space, $f:\Phi\rightarrow\mathbb{R}$ a continuous real-valued function and $t_1,t_2\in\mathbb{R}$ with $t_1>t_2$. If the interval $[t_2,t_1]$ does not contain a weak homotopy significant eigenvalue, there exists a homotopy which brings $\Phi_{\leq t_1}$ to $\Phi_{\leq t_2}$.
\end{lemma}
\begin{proof}
For all $s\in[t_2,t_1]$ there exists by assumption $h_s,l_s\in\mathbb{R}$ such that $l_s< s < h_s$, and a homotopy $H_s:\Phi_{\leq h_s}\times[0,1]\rightarrow\Phi$ which brings $\Phi_{\leq h_s}$ to $\Phi_{\leq l_s}$. The set $\{(l_s,h_s)|s\in[t_2,t_1]\}$ forms a covering of $[t_2,t_1]$. Since this interval is compact, a finite set $T\subset[0,1]$ exists such that $\{(l_s,h_s)|s\in T\}$ covers $[t_2,t_1]$. If we then define the homotopy $H:\Phi_{\leq t_1}\times[0,1]\rightarrow\Phi$ as the composition of the $H_s$ for all $s\in T$, we have found our homotopy which brings $\Phi_{\leq t_1}$ to $\Phi_{\leq t_2}$.
\end{proof}

\begin{lemma} \label{th:hssaltstable}
Let $\Phi$ be a topological space and $f:\Phi\rightarrow\mathbb{R}$ a continuous real-valued function with a weak homotopy significant eigenvalue $e\in\mathbb{R}$. If $g:\Phi\rightarrow\mathbb{R}$ is continuous and there exists a $\delta>0$ such that $\sup_{x\in\Phi}|f(x)-g(x)|<\delta$, then $g$ has a weak homotopy significant eigenvalue $\widetilde{e}$ with $|e-\widetilde{e}|<\delta$.
\end{lemma}
\begin{proof}
Choose $\delta ' > 0$ such that
\[
\sup_{x\in\Phi}|f(x)-g(x)|< \delta' <\delta
\]
Let $\epsilon > 0$ be such that $\delta' + \epsilon < \delta$.
If $e\in\mathbb{R}$ is a weak homotopy significant eigenvalue of $f$, this means in particular that we cannot bring $f^{-1}(-\infty,e + \epsilon]$ to $f^{-1}(-\infty,e-\epsilon]$. 
Then $g^{-1}(-\infty,e+\delta' + \epsilon]$ cannot be brought to $g^{-1}(-\infty,e-\delta'-\epsilon]$. Hence by Lemma \ref{th:hssaltempty}, the function $g$ contains a weak homotopy significant eigenvalue in the interval $[e-\delta'-\epsilon,e+\delta' + \epsilon]$. Therefore, the interval $(e-\delta,e+\delta)$ itself contains a weak homotopy significant eigenvalue.
\end{proof}

\begin{figure}[h!]
\centering
\begin{subfigure}{.3\linewidth}
	\centering
	\includegraphics[width=\linewidth,page=1]{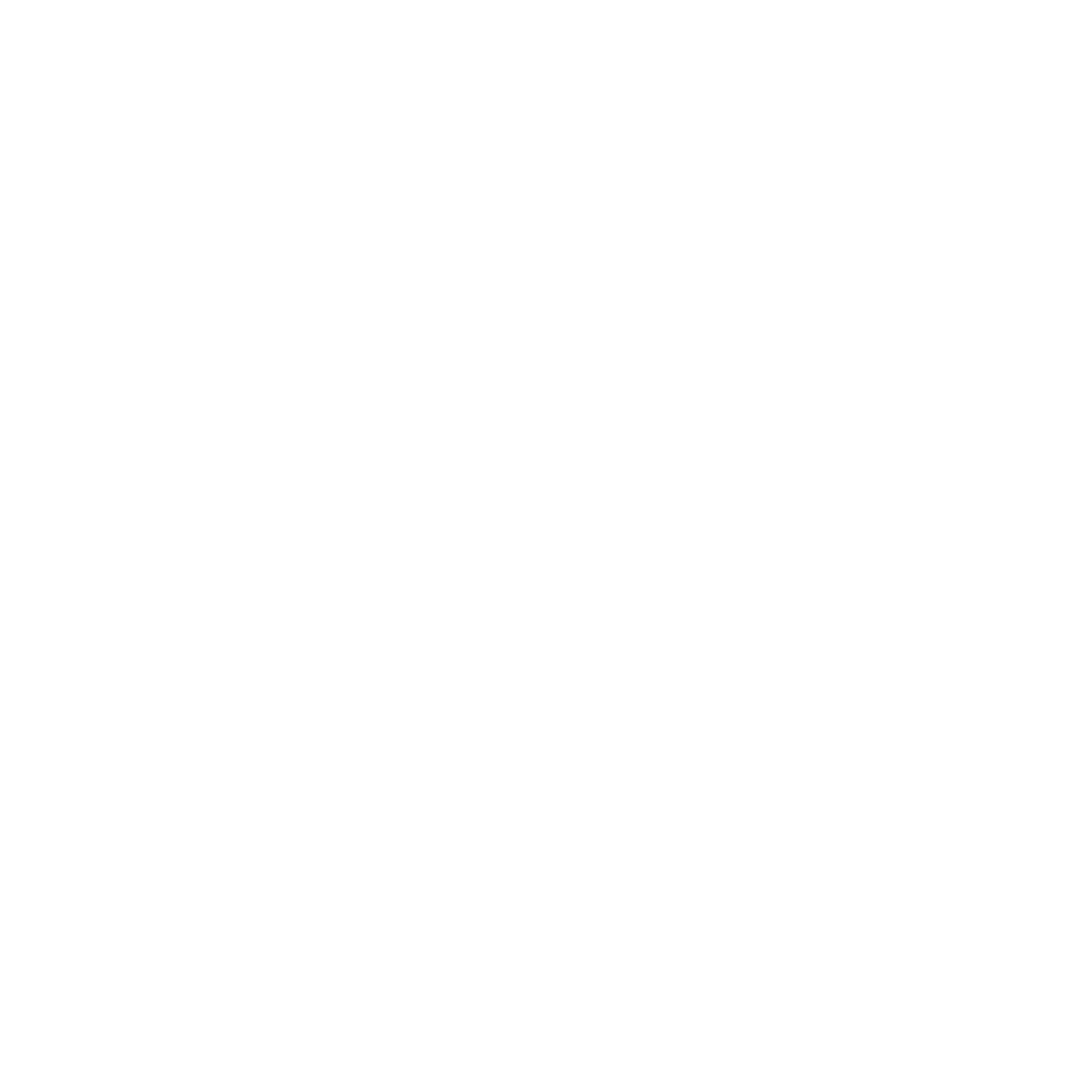}
	\caption{$V_0$, (0,3)-surgery}
\end{subfigure}
\begin{subfigure}{.3\linewidth}
	\centering
	\includegraphics[width=\linewidth,page=2]{cobordisms.pdf}
	\caption{$V_1$, (1,2)-surgery}
\end{subfigure}
\begin{subfigure}{.3\linewidth}
	\centering
	\includegraphics[width=\linewidth,page=3]{cobordisms.pdf}
	\caption{$V_2$, (1,2)-surgery}
\end{subfigure}
\par\bigskip
\begin{subfigure}{.3\linewidth}
	\centering
	\includegraphics[width=\linewidth,page=4]{cobordisms.pdf}
	\caption{$V_3$, (1,2)-surgery}
\end{subfigure}
\begin{subfigure}{.3\linewidth}
	\centering
	\includegraphics[width=\linewidth,page=5]{cobordisms.pdf}
	\caption{$V_4$, (2,1)-surgery}\label{fig:cod3}
\end{subfigure}
\begin{subfigure}{.3\linewidth} 
	\centering
	\includegraphics[width=\linewidth,page=6]{cobordisms.pdf}
	\caption{$V_5$, (2,1)-surgery} \label{fig:corpt}
\end{subfigure}
\par\bigskip
\begin{subfigure}{.3\linewidth} 
	\centering
	\includegraphics[width=\linewidth,page=7]{cobordisms.pdf}
	\caption{$V_6$, (2,1)-surgery}
\end{subfigure}
\begin{subfigure}{.3\linewidth}
	\centering
	\includegraphics[width=\linewidth,page=8]{cobordisms.pdf}
	\caption{$V_7$, (3,0)-surgery} \label{fig:corp2}
\end{subfigure}
\begin{subfigure}{.3\linewidth}
	\centering
	\includegraphics[width=\linewidth,page=1]{cobordisms.pdf}
	\caption{$V_8$}
\end{subfigure}
\caption{The sequence $(V_i)_{0\leq i\leq 8}$ of 2-dimensional submanifolds of $\rp3$. At each $V_i$ is indicated what kind of surgery should be applied to $V_i$ to obtain $V_{i+1}$. The image of the corresponding embedding of $S^{\lambda_{i+1}-1}\times D^{3-\lambda_{i+1}}$ into $V_i$ from Definition \ref{def:surgeries}, is indicated. One exception to this is the embedding of $S^1\times D^1$ in $V_4$, where we have indicated its image by ordered arrows. The outer sphere drawn from subfigure \ref{fig:cod3} onwards indicates $D^3$ on which antipodes need to be identified.} \label{fig:surgeries}
\end{figure}

\section{Examples with Morse functions}

In Section \ref{se:inequality-spectra} we provided an explicit continuous, but not differentiable, function $f$ for which the Krasnoselskii spectrum was strictly smaller than the homotopy-significant spectrum. In this section, we show that the non-differentiability of $f$ was in no way crucial. In fact, we provide examples of Morse functions with unequal spectra.

\begin{theorem}\label{th:mhsslargerks}
There exists a Morse function $F:\rp{3}\rightarrow\mathbb{R}$ such that its homotopy significant spectrum is strictly larger than its Krasnoselskii spectrum.
\end{theorem}

\begin{proof}
Recall the proof of Theorem \ref{th:hsslargerks}: there we constructed a continuous function $f:\rp3\rightarrow\mathbb{R}$ with a homotopy significant spectrum containing $0,0,0,r$ and $\sup_{x\in\rp3}f(x)$. Now take $\delta=\min\{r/2,(\sup_{x\in\rp3}f(x)-r)/2\}$. We can approximate $f$ by a Morse function $F:\rp3\rightarrow\mathbb{R}$ such that $\sup_{x\in\rp3}|F(x)-f(x)|<\delta$ \cite[Th. 2.7]{MilnorCobordism}. It then follows from applying Lemma \ref{th:hssstable}, that $F$ has at least five homotopy significant eigenvalues of which only four can be in the Krasnoselskii spectrum.
\end{proof}

In this proof of Theorem \ref{th:mhsslargerks}, we used a function for which we approximately knew some of its sublevels. This was enough to show that there was a homotopy significant eigenvalue in between two sublevels, but we did not know whether there were more (homotopy significant) critical values. If we do want a Morse function for which we know all this, we can make use of surgeries which are defined as follows \cite[Def. 3.11]{MilnorCobordism}.

\begin{definition} \label{def:surgeries}
Let $V$ be a manifold of dimension $n-1$, $\lambda\in\{0,\dotsc,n\}$ and $\phi:S^{\lambda-1} \times B^{n-\lambda}\rightarrow V$ an embedding. Define
\begin{equation}
\chi(V,\phi)=(V-\phi(S^{\lambda-1}\times \{0\}))\sqcup(B^\lambda \times S^{n-\lambda-1})/\sim
\end{equation}
where the equivalence relation identifies $\phi(u,\theta v)$ with $(\theta u,v)$ for all $u\in S^{\lambda-1}$, $v\in S^{n-\lambda-1}$ and $\theta\in(0,1)$. Then it is said that a manifold $V'$ can be obtained from $V$ by a surgery of type $(\lambda,n-\lambda)$ if $V'$ is diffeomorphic to $\chi(V,\phi)$.
\end{definition}

Surgeries enable us to construct a Morse function $F:\rp3\rightarrow\mathbb{R}$ for which we can control the amount of critical values and determine which of them are homotopy significant. We do this as follows: we construct a finite sequence of $(n-1)$-dimensional manifolds $V_0, \dots, V_8$, in which each $V_{i}$ can be obtained from $V_{i-1}$ via a surgery of type $(\lambda_i,n-\lambda_i)$ for some $\lambda_i\in\{0,\dotsc,n\}$, see Figure \ref{fig:surgeries}. Then we know by \cite[Th. 3.12]{MilnorCobordism} that for each pair of subsequent manifolds $(V_{i-1},V_i)$ a smooth $n$-dimensional manifold $N_i$ and Morse function $F_i:N_i\rightarrow\mathbb{R}$ exists such that $\partial N_i=V_{i-1}\sqcup V_i$. Moreover, $F_i$ satisfies $F_i^{-1}(0)=V_{i-1}$ and $F_i^{-1}(1)=V_i$, and $F_i$ has exactly one critical value $c_i$ which has index $\lambda_i$. Define $M_0=\emptyset$ and $M_i=\bigcup_{1\leq l\leq i} N_l$ for $i\in\{1,\dotsc,8\}$, which will form the sublevels of $F$. Then we can glue all these Morse functions $F_i$ together to form one Morse function $F:M_8 \rightarrow\mathbb{R}$ with critical values $(c_i)_{1\leq i \leq 8}$ of index $(\lambda_i)_{1\leq i \leq 8}$ \cite[Lemma 3.7]{MilnorCobordism}. Moreover, our construction is such that $M_8 = \rp3$. 

We now want to see which critical values are homotopy-significant and which are not. Hence, we need to determine for which $i$, the sublevel $M_i$ cannot be brought back to $M_{i-1}$. 

Using Proposition \ref{th:genusS} we see that the category of the sublevels increases at $c_1$, $c_4$, $c_5$ and $c_8$, which implies that these values lie in the Krasnoselskii spectrum. Furthermore, we see that $c_2$ and $c_3$ are not in the homotopy significant spectrum of $F$, because $M_2$ and $M_3$ are contractible in $\rp3$.
The value $c_6$ is homotopy significant, but is not contained in the Krasnoselskii spectrum by respectively Theorem \ref{th:rp2torp2t} and Lemma \ref{th:catrpt3}.
As $M_7$ can be brought back to $\rp2\subset M_6$, the critical value $c_7$ is in neither of the two spectra. 

All in all, the Krasnoselskii spectrum of the resulting Morse function consists of $c_1$, $c_4$, $c_5$ and $c_8$, and together with $c_6$ they form its homotopy significant spectrum.

\section*{Acknowledgments}
The authors would like to thank Slava Matveev for helpful discussions, and in particular for his suggestion that eventually has led to the example function in this article, for which the Krasnoselskii spectrum differs from the homotopy significant spectrum. The authors would like to thank the referee for helpful suggestions for improving this article.

\appendix

\section{The Cheeger constant as a non-linear eigenvalue}
\label{ap:cheeger}
\makeatletter 
\gdef\thesection{\@Alph\c@section}%
\makeatother

In this appendix we show that the Cheeger constant for a closed Riemannian manifold corresponds to the second Krasnoselskii eigenvalue. Throughout this appendix, we let $(M, g)$ be a closed Riemannian manifold and we denote by $\mu$ the standard Riemannian measure divided by the total standard Riemannian measure of $M$. 

The total variation of a function $u \in L^1(M)$ is defined as
\[
TV(u) := \sup \left\{\int_M u \ \mathrm{div} \phi \ d \mu \ | \ \phi \text{ smooth vector field on $M$ with } \| \phi \|_\infty \leq 1\right\}
\]
The space of $L^1(M)$ functions with finite total variation is called the space of functions of bounded variation, and we denote it by $BV(M)$. It is a Banach space when endowed with the norm
\[
\| u \|_{BV} := TV(u) + \| u \|_1.
\]
We will denote the unit sphere in $BV(M)$ by $S$, and the corresponding projective space by $\Phi$.

The Cheeger constant can be characterized as
\[
h_1 := \inf\{ TV(u) \ | \ u \in BV(M), \ \|u\|_1 = 1,\ \median u = 0 \}.
\]

\begin{theorem}
\label{th:cheeger-constant-second-eigenvalue}
Let $(M, g)$ be a closed Riemannian manifold. Denote by $BV(M)$ the Banach space of functions of bounded variation on $M$. Define on $BV(M)\setminus \{0\}$ the normalized energy
\[
\mathcal{E}(u) := \frac{TV(u) }{\int_M |u| d \mu}
\]
where $\mu$ is the standard Riemannian volume measure divided by the total volume of $M$.
Then the eigenvalues $\kras_2$ and $\hs_2$ of $\mathcal{E}$ equal the Cheeger constant $h_1$.
\end{theorem}

The proof of this equality of the Cheeger constant and the second eigenvalue will be easier to present if we first recall the concept of a \emph{median}.

\begin{definition}
We say that a constant $m \in \mathbb{R}$ is a \emph{median} for a function $u : M \to \mathbb{R}$, if both
\[
\mu( \{ u \leq m \} ) \geq 1/2 \qquad \text{and} \qquad \mu(\{ u \geq m \} ) \geq 1/2.
\]
\end{definition}
The function $t \mapsto \mu(\{u \leq t\})$ is right-continuous as it is the distribution function of the pushforward measure $u_{\#} \mu$, and the function $t \mapsto \mu(\{ u \geq t\})$ is left-continuous. 
These continuity properties imply that $m$ is a median for $u$ if and only if
\[
\min\{ c \in \mathbb{R} \ |\ \mu(\{ u \leq c\} ) \geq 1/2\}
\leq m \leq
 \max\{ c \in \mathbb{R} \ |\ \mu(\{ u \geq c\} )\geq 1/2 \}
\]
and in particular a median of $u$ always exists.

\begin{proposition}
\label{pr:ivt-median}
For every continuous curve $\rho : [0, 1] \to S$ between two antipodes on the unit sphere $S$ in $BV(M)$, there exists a $\sigma \in [0, 1]$ such that $0$ is a median for $\rho(\sigma)$.
\end{proposition}
\begin{proof}
Without loss of generality, we can assume that $\mu(\{ \rho(0) \leq 0 \}) \geq 1/2$.

As a step in between, we prove that
for a fixed $c \in \mathbb{R}$, the function $u \mapsto \mu( \{ u \leq c \} )$ from $L^1$ to $\mathbb{R}$ is upper semicontinuous. For that we need to show that for every sequence $(u_i)$ in $L^1$ converging to $u$ in $L^1$, it holds that
\[
\limsup_{i \to \infty}
\mu(\{ u_i \leq \delta\}) \leq \mu(\{ u \leq \delta \})
\]
Let $u_1, u_2, \dots$ be a sequence in $L^1$ converging to $u$. Let $\epsilon > 0$. Then, there exists a $\delta > 0$ such that $\mu( \{ u \leq c + \delta \} ) < \mu(\{u \leq c\}) + \epsilon$. As a consequence, for $i$ large enough, $\mu(\{ u_i \leq c \} ) < \mu(\{ u \leq c \}) + \epsilon$. This proves that $u \mapsto \mu(\{ u \leq c \})$ is upper semicontinuous. 

Similarly, the function $u \mapsto \mu( \{ u \geq c \})$ is upper semicontinuous. It follows that the functions $s \mapsto \mu( \{\rho(s) \leq 0 \} )$ and $s \mapsto \mu( \{\rho(s) \geq 0\})$ are upper semicontinuous.

We now define
\[
\sigma := \sup \{s \in [0,1] \ | \ \mu( \{ \rho(s) \leq 0 \} ) \geq 1/2\}.
\]
By upper semicontinuity, we find $\mu(\{\rho(\sigma) \leq 0\}) \geq 1/2$ and $\mu(\{\rho(\sigma) \geq 0\}) \geq 1/2$. Therefore, $0$ is a median for $\rho(\sigma)$.
\end{proof}

\begin{proof}[Proof of Theorem \ref{th:cheeger-constant-second-eigenvalue}]
We will rely on the characterization of the second eigenvalue by Lemma \ref{le:characterization-second-eigenvalue}, namely
\[
\kras_2 = \hs_2 = \inf \left\{
\sup_{s \in [0,1]} \mathcal{E}(\ell(s)) \ | \ 
\ell: [0, 1] \to \Phi \text{ non-contractible loop} 
\right\}
\]

We first show that $h_1 \leq \kras_2$. Fix an arbitrary non-contractible loop $\ell: [0,1] \to \Phi$. 
Then $\ell$ lifts to a curve $\rho: [0,1] \to S$ between two antipodes on the unit sphere $S$ in $BV(M)$.
By Proposition \ref{pr:ivt-median}, there is a $\sigma \in [0, 1]$ such that $0$ is a median for $\ell(\sigma)$.
Therefore,
\[
h_1 \leq \mathcal{E}(\ell(\sigma)) \leq \sup_{s \in [0,1]} \mathcal{E}(\ell(s))
\]
As $\ell$ was an arbitrary non-contractible loop, we find that $h_1 \leq \kras_2$.

We now show that $\kras_2 \leq h_1$.
Let $u$ be a function in $BV(M)$ such that $\ \| u \|_1 = 1$ and $0$ is a median for $u$.
Consider the curve $\ell :[-\pi/2, \pi/2] \to \Phi$ given by 
\[
\ell(s) := [\tan(s) + u]
\]
for $s \in (-\pi/2, \pi/2)$ and extended continuously to the boundary. That is, $\ell(-\pi/2)$ and $\ell(\pi/2)$ equal the equivalence class of the constant, non-zero function. Then $\ell$ is non-contractible in $\Phi$, and for all $s \in [-\pi/2, \pi/2]$, it holds that $\mathcal{E}(\ell(s)) \leq TV(u)$, where we used that a number $m \in \mathbb{R}$ is a median of $u$ if $c = m$ minimizes
\[
\int_M | u - c | d \mu.
\]
Therefore
\[
\kras_2 \leq TV(u)
\]
and as $u$ was arbitrary with $\|u\| = 1$ and $0 \in \median u$, we find $\kras_2 \leq h_1$.
\end{proof}

\bibliographystyle{plain}
\bibliography{Comparison}

\begin{thebibliography}{10}

\bibitem{Ambrosio-New-2017}
Luigi Ambrosio and Shouhei Honda.
\newblock New stability results for sequences of metric measure spaces with
  uniform {Ricci} bounds from below.
\newblock In {\em Measure Theory in Non-Smooth Spaces}, pages 1--51. De
  Gruyter, Warsaw, 2017.

\bibitem{Ambrosio-Continuity-2018}
Luigi Ambrosio, Shouhei Honda, and Jacobus Portegies.
\newblock Continuity of nonlinear eigenvalues in {CD}({K},$\infty$) spaces with
  respect to measured {Gromov}--{Hausdorff} convergence.
\newblock {\em Calculus of Variations and Partial Differential Equations},
  57(2):34, 2018.

\bibitem{Belkin-Laplacian-2003}
Mikhail Belkin and Partha Niyogi.
\newblock Laplacian eigenmaps for dimensionality reduction and data
  representation.
\newblock {\em Neural computation}, 15(6):1373--1396, 2003.

\bibitem{Burger-Spectral-2016}
Martin Burger, Guy Gilboa, Michael Moeller, Lina Eckardt, and Daniel Cremers.
\newblock Spectral decompositions using one-homogeneous functionals.
\newblock {\em SIAM Journal on Imaging Sciences}, 9(3):1374--1408, 2016.

\bibitem{Burton-Finding-2017}
Benjamin Burton, Arnaud de~Mesmay, and Uli Wagner.
\newblock Finding non-orientable surfaces in 3-manifolds.
\newblock {\em Discrete \& Computational Geometry}, 58(4):871--888, 2017.

\bibitem{Calabi-Linear-1964}
Eugenio Calabi.
\newblock Linear systems of real quadratic forms.
\newblock {\em Proceedings of the American Mathematical Society},
  15(5):844--846, 1964.

\bibitem{Chang-Spectrum-2016}
Kung~Ching Chang.
\newblock Spectrum of the 1-laplacian and cheeger's constant on graphs.
\newblock {\em Journal of Graph Theory}, 81(2):167--207, 2016.

\bibitem{Coifman-Diffusion-2006}
Ronald Coifman and St{\'e}phane Lafon.
\newblock Diffusion maps.
\newblock {\em Applied and computational harmonic analysis}, 21(1):5--30, 2006.

\bibitem{Duits-Total-2019}
Remco Duits, Etienne St-Onge, Jim Portegies, and Bart Smets.
\newblock Total variation and mean curvature {PDEs} on the space of positions
  and orientations.
\newblock In {\em International Conference on Scale Space and Variational
  Methods in Computer Vision}, pages 211--223. Springer, 2019.

\bibitem{Gilboa-Total-2014}
Guy Gilboa.
\newblock A total variation spectral framework for scale and texture analysis.
\newblock {\em SIAM journal on Imaging Sciences}, 7(4):1937--1961, 2014.

\bibitem{Gromov-Dimension}
Misha Gromov.
\newblock Dimension, non-linear spectra and width.
\newblock In {\em Geometric Aspects of Functional Analysis}, pages 132--184.
  Springer, Berlin, 1988.

\bibitem{Gromov-Pauli-2009}
Misha Gromov.
\newblock Geometry, topology and spectra of non-linear spaces of maps.
\newblock {\em ETH Wolfgang Pauli Lectures}, 2009.
\newblock Available at \texttt{https://video.ethz.ch/speakers/pauli/2009.html}.

\bibitem{Gromov2015}
Misha Gromov.
\newblock Morse spectra, homology measures and parametric packing problems.
\newblock {\em arXiv preprint arXiv:1710.03616}, 2017.

\bibitem{Littig-Convergence-2014}
Samuel Littig and Friedemann Schuricht.
\newblock Convergence of the eigenvalues of the $p$-laplace operator as $p$
  goes to 1.
\newblock {\em Calculus of Variations and Partial Differential Equations},
  49(1-2):707--727, 2014.

\bibitem{MilnorCobordism}
John Milnor.
\newblock {\em {On the h-cobordism theorem}}.
\newblock Princeton University Press, Princeton, NJ, 1st edition, 1965.

\bibitem{MilnorMorse}
John Milnor.
\newblock {\em Morse Theory}.
\newblock Princeton University Press, Princeton, NJ, 5th edition, 1973.

\bibitem{Parini-Second-2010}
Enea Parini.
\newblock The second eigenvalue of the $p$-{L}aplacian as $p$ goes to 1.
\newblock {\em International Journal of Differential Equations}, Article--ID
  984671, 2010.

\bibitem{Rabinowitz1973}
Paul Rabinowitz.
\newblock Some aspects of nonlinear eigenvalue problems.
\newblock {\em Rocky Mountain Journal of Mathematics}, 3(2):161--202, 1973.

\bibitem{Rudin-Nonlinear-1992}
Leonid Rudin, Stanley Osher, and Emad Fatemi.
\newblock Nonlinear total variation based noise removal algorithms.
\newblock {\em Physica D: nonlinear phenomena}, 60(1-4):259--268, 1992.

\bibitem{Shen-non-linear-1998}
Zhongmin Shen.
\newblock The non-linear {Laplacian} for {Finsler} manifolds.
\newblock In {\em The theory of Finslerian Laplacians and applications}, pages
  187--198. Springer, Dordrecht, 1998.

\bibitem{Struwe}
Michael Struwe.
\newblock {\em Variational Methods}, volume~34.
\newblock Springer-Verlag, Berlin, 4th edition, 2008.

\bibitem{Zeune-Multiscale-2017}
Leonie Zeune, Guus van Dalum, Leon Terstappen, Stephan van Gils, and Christoph
  Brune.
\newblock Multiscale segmentation via {Bregman} distances and nonlinear
  spectral analysis.
\newblock {\em SIAM journal on imaging sciences}, 10(1):111--146, 2017.

\end{thebibliography}

\end{document}